\documentclass[11pt]{article}
\usepackage{amsmath,amsthm}
\usepackage{amssymb}
\usepackage{enumerate}
\usepackage[mathscr]{eucal}
\usepackage[cm]{fullpage}
\usepackage[english]{babel} 
\usepackage[latin1]{inputenc}
\usepackage[dvipsnames]{xcolor} %
\newcommand{\T}{\mathscr{T}}
\newcommand{\IO}{\mathscr{IO}}
\newcommand{\On}{\mathscr{O}_n}
\newcommand{\wEnd}{\mathrm{wEnd}}
\theoremstyle{plain}
\newtheorem{theorem}{Theorem}[section]
\newtheorem{proposition}[theorem]{Proposition}
\newtheorem{lemma}[theorem]{Lemma}
\newtheorem{corollary}[theorem]{Corollary}
\def\im{\mathop{\mathrm{Im}}\nolimits} 
\def\N{\mathbb N}
\def\<{\leqslant}
\def\>{\geqslant}
\begin{document}

\title{On the monoid of order-preserving transformations of a finite chain whose ranges are intervals} 

\author{V\'\i tor H. Fernandes\footnote{This work is funded by national funds through the FCT - Funda\c c\~ao para a Ci\^encia e a Tecnologia, I.P., under the scope of the projects UIDB/00297/2020 (https://doi.org/10.54499/UIDB/00297/2020) and UIDP/00297/2020 (https://doi.org/10.54499/UIDP/00297/2020) (Center for Mathematics and Applications).}}

\maketitle

\vspace*{-4em}

\renewcommand{\thefootnote}{}



\renewcommand{\thefootnote}{\arabic{footnote}}
\setcounter{footnote}{0}

\begin{abstract} 
In this note we give a presentation for the monoid 
$\IO_n$ of all order-preserving transformations of a $n$-chain whose ranges are intervals. 
We also consider the submonoid $\IO_n^-$ of $\IO_n$ consisting of order-decreasing transformations, for which we determine the cardinality, the rank and a presentation. 
\end{abstract}

\section{Introduction}\label{intro}  

Let $n\in\N$ and $\Omega_n=\{1,\ldots,n\}$. Denote by $\T_n$ the monoid (under composition) of all full transformations of $\Omega_n$. 
From now on, we consider $\Omega_n$ equipped with the usual linear order $1<\cdots<n$. 
Recall that a transformation $\alpha$ of $\Omega_n$ is said to be \textit{order-preserving} 
if $x\leqslant y$ implies $x\alpha\leqslant y\alpha$, for all $x,y\in\Omega_n$. 
Denote by $\On$ the submonoid of $\T_n$ of all order-preserving transformations.   

\smallskip 

Semigroups of order-preserving transformations have been studied extensively for over six decades. 
The oldest papers known to the author date back to 1962, are by A\v\i zen\v stat \cite{Aizenstat:1962,Aizenstat:1962b}  and give a presentation and a description of the congruences of $\On$.  
After that and up to the present day, many other papers have been published studying various properties of $\On$. 
See, for example, 
\cite{Fernandes:1997,
Fernandes&al:2010, 
Fernandes&Volkov:2010,
Gomes&Howie:1992, 
Higgins:1995,
Howie:1971,
Laradji&Umar:2006,
Vernitskii&Volkov:1995}. 
Nonetheless, these semigroups still reveal some mystery. The problem of decidability of the pseudovariety generated by the family $\On$, with $n\in\N$, proposed by  J.-E. Pin in the ``Szeged International Semigroup Colloquium" (1987), is perhaps the oldest and most important that remains to be solved. 

\smallskip 

The main object of this work is the submonoid $\IO_n$ of $\On$ consisting of all transformations $\alpha\in\On$ such that the image of $\alpha$ 
is an interval of $\Omega_n$, i.e. $x\<y\<z$ and $x,z\in\im(\alpha)$ imply $y\in\im(\alpha)$, for all $x,y,z\in\Omega_n$. 

The monoid $\IO_n$ is not only a very natural submonoid of $\On$ but also coincides with the monoid of all \textit{weak endomorphisms} 
of a directed path with $n$ vertices. See \cite{Fernandes&Paulista:2023}. 
In the aforementioned paper, the authors studied the regularity, determined the rank, the cardinality and the number of idempotents of $\IO_n$ 
(denoted there by $\wEnd\vec{P}_n$). 
In this note, our main aim is to exhibit a presentation for $\IO_n$. 
We also consider its submonoids $\IO_n^-$ consisting of order-decreasing transformations  
and $\IO_n^+$ consisting of order-increasing transformations, 
for which we determine the cardinalities, the ranks and presentations.

\medskip 

Recall that the \textit{rank} of a (finite) monoid $M$ is the minimum size that a generating set of $M$ can have. 

Next, we also recall some notions related to the concept of a monoid presentation. 

Let $A$ be an alphabet and consider the free monoid $A^*$ generated by $A$.
The elements of $A$ and of $A^*$ are called \textit{letters} and \textit{words}, respectively. 
The empty word is denoted by $1$ and we write $A^+$ to express $A^*\setminus\{1\}$. 
The number of letters of a word $w\in A^*$ is called the \textit{length} of $w$ and it is denoted by $|w|$.  
A pair $(u,v)$ of $A^*\times A^*$ is called a
\textit{relation} of $A^*$ and it is usually represented by $u=v$. 
A \textit{monoid presentation} is an ordered pair
$\langle A\mid R\rangle$, where $R\subseteq A^*\times A^*$ is a set of relations of
the free monoid $A^*$. 
Denote by $\sim_R$ the smallest congruence on $A^*$ containing $R$. 
A monoid $M$ is said to be \textit{defined by a presentation} $\langle A\mid R\rangle$ if $M$ is
isomorphic to $A^*/\sim_R$.  
Let $X$ be a generating set of a monoid $M$ and let $\phi: A\longrightarrow M$ be an injective mapping
such that $A\phi=X$.
Let $\varphi: A^*\longrightarrow M$ be the (surjective) homomorphism of monoids that extends $\phi$ to $A^*$.
We say that $X$ satisfies (via $\varphi$) a relation $u=v$ of $A^*$ if $u\varphi=v\varphi$. 
For more details see \cite{Ruskuc:1995}.

Remember that, a usual method to find a presentation for a finite monoid
is described by the following result (adapted to the monoid case from
\cite[Proposition 3.2.2]{Ruskuc:1995}). 

\begin{proposition}[Guess and Prove method] \label{ruskuc} 
Let $M$ be a finite monoid generated by a set $X$, 
let $A$ be an alphabet and let $\phi: A\longrightarrow M$ be an injective mapping
such that $A\phi=X$. 
Let $\varphi:A^*\longrightarrow M$ be the (surjective) homomorphism of monoids 
that extends $\phi$ to $A^*$, let $R\subseteq A^*\times A^*$ 
and $W\subseteq A^*$. Assume that the following conditions are
satisfied:
\begin{enumerate}
\item The generating set $X$ of $M$ satisfies (via $\varphi$) all relations from $R$;
\item For each word $w\in X^*$, there exists a word $w'\in W$ such that $w\sim_R w'$;
\item $|W|\leqslant |M|$.
\end{enumerate}
Then, $M$ is defined by the presentation $\langle A\mid R\rangle$. 
\end{proposition}

Notice that, if $W$ satisfies the above conditions then, in fact,
$|W|=|M|$. 

\medskip 

For general background on Semigroup
Theory and standard notation, we refer the reader to Howie's book \cite{Howie:1995}. 
We would also like to mention the use of computational tools, namely GAP \cite{GAP4}.

\section{Preliminaries}\label{pre} 

We start this section by recalling the cardinality, the rank and generators of $\IO_n$ that have been established in \cite{Fernandes&Paulista:2023}. 
By \cite[Theorem 2.6]{Fernandes&Paulista:2023}, we have $|\IO_n|=\sum_{k=1}^{n}(n-k+1)\binom{n-1}{k-1}$, whence 
$$
|\IO_n|=(n+1)2^{n-2}.
$$ 
Next, for $n\geqslant 2$, define 
$$
a_{i}=
\begin{pmatrix} 
1 & \cdots & i & i+1 & \cdots & n \\ 
1 & \cdots & i & i & \cdots & n-1\\
\end{pmatrix}
\quad\text{and}\quad
b_{i}=
\begin{pmatrix} 
1 & \cdots & i & i+1 & \cdots & n \\ 
2 & \cdots & i+1 & i+1 & \cdots & n\\
\end{pmatrix},
$$
for $1\<i\<n-1$. Clearly, $a_{1},\ldots ,a_{n-1}, b_{1},\ldots ,b_{n-1}\in \IO_n$. 
Moreover, by \cite[Proposition 3.3 and Theorem 3.5]{Fernandes&Paulista:2023}, for $n\>3$, 
$\{ a_{1},\ldots ,a_{n-2},b_{n-1} \}$ is a generating set of $\IO_n$ of minimum size. 
In particular, for $n\>3$, the monoid $\IO_n$ has rank $n-1$. 
Notice that $\IO_1=\T_1$ and 
$\IO_2=\left\{\left(\begin{smallmatrix} 1 &2 \\ 1 &1 \\ \end{smallmatrix}\right),\left(\begin{smallmatrix} 1 &2 \\ 1 &2 \\ \end{smallmatrix}\right),
\left(\begin{smallmatrix} 1 &2 \\ 2 &2 \\ \end{smallmatrix}\right)\right\}$ and so, clearly, $\IO_1$ has rank $0$ and $\IO_2$ has rank $2$. 

\smallskip 

Now, observe that, for $1\leqslant i,j\leqslant n-1$, we have 
$$
a_ia_j = \left\{
\begin{array}{ll}
\begin{pmatrix} 
1 & \cdots & i & i+1 & \cdots &j+1 &j+2&\cdots & n \\ 
1 & \cdots & i & i & \cdots & j& j& \cdots & n-2\\
\end{pmatrix} & \mbox{if $i\leqslant j$} \\
\begin{pmatrix} 
1 & \cdots & j & j+1 & \cdots &i &i+1&\cdots & n \\ 
1 & \cdots & j & j & \cdots & i-1& i-1& \cdots & n-2\\
\end{pmatrix} & \mbox{if $i > j$}, 
\end{array}
\right. 
$$
$$
a_ib_j = \left\{
\begin{array}{ll}
\begin{pmatrix} 
1 & \cdots & i & i+1 & \cdots &j+1 &j+2&\cdots & n \\ 
2 & \cdots & i+1 & i+1 & \cdots & j+1& j+1& \cdots & n-1\\
\end{pmatrix} & \mbox{if $i\leqslant j$} \\
\begin{pmatrix} 
1 & \cdots & j & j+1 & \cdots &i &i+1&\cdots & n \\ 
2 & \cdots & j+1 & j+1 & \cdots & i& i& \cdots & n-1\\
\end{pmatrix} & \mbox{if $i > j$}, 
\end{array}
\right. 
$$
$$
b_ib_j = \left\{
\begin{array}{ll}
\begin{pmatrix} 
1 & \cdots & i & i+1 & \cdots &j &j+1&\cdots & n \\ 
3 & \cdots & i+2 & i+2& \cdots & j+1& j+1& \cdots & n\\
\end{pmatrix} & \mbox{if $i<j$} \\
\begin{pmatrix} 
1 & \cdots & j-1 & j & \cdots &i &i+1&\cdots & n \\ 
3 & \cdots & j +1& j+1 & \cdots & i+1& i+1& \cdots & n\\
\end{pmatrix} & \mbox{if $i \geqslant j$} 
\end{array}
\right. 
$$
and 
$$
b_ia_j = \left\{
\begin{array}{ll}
\begin{pmatrix} 
1 & \cdots & i & i+1 & \cdots &j &j+1&\cdots & n \\ 
2 & \cdots & i+1 & i+1& \cdots & j& j& \cdots & n-1\\
\end{pmatrix} & \mbox{if $i<j$} \\
\begin{pmatrix} 
1 & \cdots & j-1 & j & \cdots &i &i+1&\cdots & n \\ 
2 & \cdots & j & j & \cdots & i& i& \cdots & n-1\\
\end{pmatrix} & \mbox{if $i \geqslant j$} , 
\end{array}
\right. 
$$
from which we can deduce the equalities 
\begin{equation}\label{eq1}
\mbox{$a_ia_{n-1}=a_i=b_ia_1$~ and ~$b_ib_1=b_i=a_ib_{n-1}$,~ for $1\leqslant i\leqslant n-1$}, 
\end{equation}
\begin{equation}\label{eq2}
\mbox{$a_ia_j=a_{j+1}a_i$~ and ~$b_ib_{j+1}=b_{j+1}b_{i+1}$,~ for $1\leqslant i\leqslant j\leqslant n-2$} 
\end{equation}
and 
\begin{equation}\label{eq3}
\mbox{$b_ia_j=a_ib_{j-1}=b_ja_{i+1}=a_jb_i$,~ for $1\leqslant i< j\leqslant n-1$}. 
\end{equation}

Let $A=\{a_{1},\ldots ,a_{n-1}\}$ and $B=\{b_{1},\ldots ,b_{n-1}\}$.  
Consider the alphabet $A\cup B$ with $2n-2$ letters and the set $R$ formed by the following 
$\frac{1}{2}(3n^2-n-2)$ monoid relations:  

\begin{description}

\item $(R_1)$ $a_ia_{n-1}=a_i$,~ $1\leqslant i\leqslant n-1$; 

\item $(R_2)$ $a_ia_j=a_{j+1}a_i$,~ $1\leqslant i\leqslant j\leqslant n-2$; 

\item $(R_3)$ $b_ib_1=b_i$,~ $1\leqslant i\leqslant n-1$; 

\item $(R_4)$ $b_ib_{j+1}=b_{j+1}b_{i+1}$,~ $1\leqslant i\leqslant j\leqslant n-2$; 

\item $(R_5)$ $b_ia_j=a_jb_i$,~ $1\leqslant i< j\leqslant n-1$; 

\item $(R_6)$ $b_ia_1=a_i$,~ $1\leqslant i\leqslant n-1$;  

\item $(R_7)$ $a_ib_{n-1}=b_i$,~ $1\leqslant i\leqslant n-1$. 

\end{description}

Our main aim in this note is to show that $\IO_n$ is defined by the presentation $\langle a_{1},\ldots ,a_{n-1}, b_{1},\ldots ,b_{n-1}\mid R\rangle$. 

\smallskip 

Let $\varphi:(A\cup B)^*\longrightarrow\IO_n$ be the homomorphism of monoids that extends the mapping $A\cup B\longrightarrow\IO_n$ defined by $a_i\mapsto a_i$ and 
$b_i\mapsto b_i$, for $1\<i\<n-1$. 
Therefore, taking into account the above equalities (\ref{eq1})-(\ref{eq3}), 
we can already conclude that the generating set $A\cup B$ of $\IO_n$ satisfies (via $\varphi$) all relations from $R$. 

\section{Decreasing and increasing transformations}

Let us consider the following submonoids of $\T_n$: 
$$
\T_n^-=\{\alpha\in\T_n\mid \mbox{$x\alpha\leqslant x$, for all $x\in\Omega_n$}\} \; (\mbox{order-decreasing transformations of $\Omega_n$})
$$
and 
$$
\T_n^+=\{\alpha\in\T_n\mid \mbox{$x\leqslant x\alpha$, for all $x\in\Omega_n$}\} \; (\mbox{order-increasing transformations of $\Omega_n$}). 
$$

Consider also the mapping $\varphi:\T_n\longrightarrow\T_n$ which maps each transformation $\alpha\in\T_n$ into the transformation $\bar\alpha\in\T_n$ defined by 
$(x)\bar\alpha=n+1-(n+1-x)\alpha$, for all $x\in\Omega_n$. It is a routine matter to check that $\varphi$ is an automorphism of monoids such that 
$\varphi^2$ is the identity mapping of $\T_n$, $\T_n^-\varphi\subseteq\T_n^+$,  $\T_n^+\varphi\subseteq\T_n^-$ and $\IO_n\varphi\subseteq\IO_n$. 
Consequently, 
$$
\T_n^-\varphi=\T_n^+, \quad \T_n^+\varphi=\T_n^- \quad\text{and}\quad \IO_n\varphi=\IO_n. 
$$

Now, consider the submonoids $\IO_n^-=\IO_n\cap\T_n^-$ and $\IO_n^+=\IO_n\cap\T_n^+$ of $\IO_n$. It follows immediately from the previous properties that 
$$
\IO_n^-\varphi=\IO_n^+  \quad\text{and}\quad \IO_n^+\varphi=\IO_n^-. 
$$ 
Therefore, $\IO_n^-$ and $\IO_n^+$ are isomorphic submonoids of $\IO_n$. 

Next, observe that $a_1,\ldots,a_{n-1}\in\IO_n^-$ and $b_1,\ldots,b_{n-1}\in\IO_n^+$. 
Moreover, $\bar a_i=b_{n-i}$ and $\bar b_i=a_{n-i}$, for $1\leqslant i\leqslant n-1$. 

The following property gives a characterization of $\IO_n^-$ that allows us to easily count its number of elements. 

\begin{proposition}\label{char-}
For $n\geqslant1$, $\IO_n^-=\{\alpha\in\IO_n\mid\mbox{\em$\im(\alpha)=\{1,\ldots,k\}$, for some $1\leqslant k\leqslant n$}\}$. 
\end{proposition} 
\begin{proof}
Let $\alpha\in\IO_n^-$. Then, $1\alpha=1$ and so $\im(\alpha)=\{1,\ldots,k\}$, for some $1\leqslant k\leqslant n$. 

Conversely, let $\alpha\in\IO_n$ be such that $\im(\alpha)=\{1,\ldots,k\}$, for some $1\leqslant k\leqslant n$. 
Then, $1\alpha=\min\im(\alpha)=1$. Assume by induction hypothesis that $t\alpha\leqslant t$, for (a fixed) $1\leqslant t<n$. 
By contradiction, suppose that $(t+1)\alpha>t+1$. Then, $t\alpha\leqslant t<t+1<(t+1)\alpha\leqslant k$, 
whence $t+1\in\im(\alpha)$ and so $t+1=x\alpha$, for some $x\in\Omega_n$. 
If $x\leqslant t$, then $t+1=x\alpha\leqslant t\alpha\leqslant t$, which is a contradiction. 
Hence, $t+1\leqslant x$ and so $t+1<(t+1)\alpha\leqslant x\alpha=t+1$, which is again a contradiction. 
So, $(t+1)\alpha\leqslant t+1$. Thus, $t\alpha\leqslant t$, for all $t\in\Omega_n$, and so $\alpha\in\IO_n^-$, as required. 
\end{proof}

Of course, from the last proposition and the equality $\IO_n^+=\IO_n^-\varphi$, it follows that 
$$
\IO_n^+=\{\alpha\in\IO_n\mid\mbox{$\im(\alpha)=\{k,\ldots,n\}$, for some $1\leqslant k\leqslant n$}\}. 
$$

\begin{proposition}\label{size-}
For $n\geqslant1$, $|\IO_n^-|=|\IO_n^+|=2^{n-1}$. 
\end{proposition}
\begin{proof}
Let $1\leqslant k\leqslant n$. 
In view of Proposition \ref{char-}, we have 
$$
\{\alpha\in\IO_n^-\mid |\im(\alpha)|=k\}=\{\alpha\in\IO_n^-\mid \im(\alpha)=\{1,\ldots,k\}\}. 
$$
Then, for $\IO_n^-$, we can adapt the proof of \cite[Theorem 1.5]{Fernandes&Paulista:2023} for $\IO_n$, taking into account that instead of $n-k+1$ possibilities for images with $k$ elements, we only have one possibility. 
In this way, we get 
$$
|\{\alpha\in\IO_n^-\mid |\im(\alpha)|=k\}=\binom{n-1}{k-1}. 
$$

Thus, $|\IO_n^-|=\sum_{k=1}^{n}\binom{n-1}{k-1}=\sum_{k=0}^{n}\binom{n-1}{k}=2^{n-1}$,
as required. 
\end{proof}

\smallskip 

Now, observe that $\IO_1^-=\IO_1=\T_1$ and 
$\IO_2^-=\left\{\left(\begin{smallmatrix} 1 &2 \\ 1 &1 \\ \end{smallmatrix}\right),\left(\begin{smallmatrix} 1 &2 \\ 1 &2 \\ \end{smallmatrix}\right)\right\}$. 
Then, clearly, $\IO_1^-$ has rank $0$ and $\IO_2^-$ has rank $1$. 
More generally, in what follows,  for $n\geqslant1$, we aim to show that $A$ is a generating set with minimum size of the monoid $\IO_n^-$.  

First, notice that $\{\alpha\in\IO_n^-\mid |\im(\alpha)|=n-1\}=A$ and, of course, like for $\IO_n$, the group of units of $\IO_n^-$ is trivial. 
On the other hand, we have: 

\begin{lemma}\label{ger-}
For $n\geqslant3$, let $\alpha\in\IO_n^-$ be such that $|\im(\alpha)|=k$, for some $1\leqslant k\leqslant n-2$. 
Then, there exist $\gamma_1,\gamma_2\in\IO_n^-$ such that $|\im(\gamma_1)|=|\im(\gamma_2)|=k+1$ and $\alpha=\gamma_1\gamma_2$. 
\end{lemma}
\begin{proof}
It suffices to observe that in the proof of \cite[Lemma 3.2]{Fernandes&Paulista:2023}, for an element $\alpha\in\IO_n^-$, we must have $j=0$ 
and thus the transformations $\gamma_1$ and $\gamma_2$ there defined also belong to $\IO_n^-$. 
\end{proof}

Therefore, by a simple inductive reasoning, we can conclude that $\IO_n^-$ is generated by $A$. 
In addition, by observing that $a_1,\ldots,a_{n-1}$ are the only elements of $\IO_n^-$ with images of size $n-1$ and their kernels are pairwise distinct, it is immediate to deduce that $a_1,\ldots,a_{n-1}$ are undecomposable elements (i.e. they are not a product of two elements distinct from themselves) of $\IO_n^-$, and so they must belong to any generating set of $\IO_n^-$. Thus, we get: 

\begin{theorem}\label{rank-}
For $n\geqslant1$, $A$ is a generating set with minimum size of the monoid $\IO_n^-$. 
In particular, the monoid $\IO_n^-$ has rank $n-1$. 
\end{theorem}

Since $\IO_n^+=\IO_n^-\varphi$ and $\{b_1,\ldots,b_{n-1}\}=\{a_1,\ldots,a_{n-1}\}\varphi$, 
we also get that $B$ is a generating set with minimum size of the monoid $\IO_n^+$. 

\medskip 

We finish this section by exhibiting a presentation for $\IO_n^-$. 

Let $R^-$ be the set of $\frac{1}{2}(n^2-n)$ relations $R_1\cup R_2$, considered on the alphabet $A$. 
Let $\varphi^-:A^*\longrightarrow\IO_n$ be the homomorphism of monoids that extends the mapping $A\longrightarrow\IO_n^-$ defined by $a_i\mapsto a_i$, for $1\<i\<n-1$. 
Clearly, all relations from $R^-$ are satisfied (via $\varphi^-$) by the generators $A$ of $\IO_n^-$. 

Let 
$$
W^-=\{a_{i_1}\cdots a_{i_k}\in A^*\mid \mbox{$0\leqslant k\leqslant n-1$ and $1\leqslant i_k<\cdots<i_1\leqslant n-1$}\}. 
$$
It is clear that $|W^-|=2^{n-1}$ (as there is an obvious bijection between $W^-$ and the power set of $\{1,\ldots,n-1\}$) and $W^-$ contains the empty word (taking $k=0$) and $A$. Moreover, we have: 

\begin{lemma}\label{can-}
Let $w\in A^*$. Then, there exists a word $w'\in W^-$ such that $|w'|\<|w|$ and $w\sim_{R^-}w'$. 
\end{lemma}
\begin{proof}
Let us denote $\sim_{R^-}$ simply by $\sim$. 

If $|w|=1$, then $w\in W^-$. 

Next, suppose that $w=a_{j_1}a_{j_2}$, with $1\leqslant j_1,j_2\leqslant n-1$. 
If $j_2<j_1$, then $w\in W^-$. So, suppose that $j_1\leqslant j_2$. 
If $j_2=n-1$, then $w=a_{j_1}a_{n-1}\sim a_{j_1}\in W^-$. 
On the other hand, if $j_2\leqslant n-2$, then $w=a_{j_1}a_{j_2}\sim a_{j_2+1}a_{j_1}\in W^-$, since $j_1<j_2+1$. 

Now, assume by induction hypothesis that the property is valid for words $w\in A^*$ of length less than or equal to $\ell$, for a given (fixed) $\ell\>2$. 

Let $w=a_{j_1}\cdots a_{j_\ell}a_{j_{\ell+1}}$, for some $1\leqslant j_1,\ldots,j_{\ell+1}\leqslant n-1$. 
Then, by induction hypothesis, $w\sim a_{i_1}\cdots a_{i_k}a_{j_{\ell+1}}$, with $1\<k\<\ell$ and $1\leqslant i_k<\cdots<i_1\leqslant n-1$. 

If $j_{\ell+1}=n-1$, then $w\sim a_{i_1}\cdots a_{i_k}\in W^-$. If $j_{\ell+1}<i_k$, then $w\sim a_{i_1}\cdots a_{i_k}a_{j_{\ell+1}}\in W^-$. 
So, suppose that $i_k\<j_{\ell+1} <n-1$. 
Hence, we have 
$$
w\sim a_{i_1}\cdots a_{i_k}a_{j_{\ell+1}}\sim a_{i_1}\cdots a_{i_{k-1}} a_{j_{\ell+1}+1}a_{i_k} 
\sim a_{i_1}\cdots a_{i_{k-t}} a_{j_{\ell+1}+t}a_{i_{k-t+1}}\cdots a_{i_k},
$$ 
for all $1\<t\<k$ such that $i_{k-t+1}\<j_{\ell+1}+t-1\<n-2$. 
Let 
$
t=\max\{1\<t\<k\mid i_{k-t+1}\<j_{\ell+1}+t-1\<n-2\}.
$
If $t=k$ (i.e. $i_{1}\<j_{\ell+1}+k-1\<n-2$), then $w\sim a_{j_{\ell+1}+k}a_{i_{1}}\cdots a_{i_k} \in W^-$, since $j_{\ell+1}+k>i_1$. 
Therefore, suppose that $t<k$. In this case, we have $j_{\ell+1}+t<i_{k-t}$ or $j_{\ell+1}+t=n-1$. 
If $j_{\ell+1}+t=n-1$, then $w\sim a_{i_1}\cdots a_{i_{k-t}} a_{i_{k-t+1}}\cdots a_{i_k}\in W^-$. 
On the other hand, if $j_{\ell+1}+t<i_{k-t}$, then 
$w\sim a_{i_1}\cdots a_{i_{k-t}} a_{j_{\ell+1}+t}a_{i_{k-t+1}}\cdots a_{i_k}\in W^-$, 
since $1\leqslant i_k<\cdots<  i_{k-t+1} < j_{\ell+1}+t<i_{k-t} <\cdots < i_1\leqslant n-1$. 

We finish this proof by observing that, in each case, we obtained $w\sim w'\in W^-$ with $|w'|\<k+1\<\ell+1=|w|$, 
as required. 
\end{proof} 

At this moment, we gather all the ingredients to, given Proposition \ref{ruskuc}, conclude:

\begin{theorem}\label{pres-}
For $n\>1$, the monoid $\IO_n^-$ is defined by the presentation $\langle A\mid R^-\rangle$ on $n-1$ generators and  $\frac{1}{2}(n^2-n)$ relations. 
\end{theorem} 

Since $\IO_n^+=\IO_n^-\varphi$ and $\bar a_i=b_{n-i}$ and $\bar b_i=a_{n-i}$, for $1\leqslant i\leqslant n-1$, being $R^+=R_3\cup R_4$, 
we can also conclude that the monoid $\IO_n^+$ is defined by the presentation $\langle B\mid R^+\rangle$ 
on $n-1$ generators and  $\frac{1}{2}(n^2-n)$ relations. 

\smallskip 

For reference below, observe also that, a property for $\langle B\mid R^+\rangle$ analogous to Lemma \ref{can-} is also valid: 
let 
$$
W^+=\{b_{j_1}\cdots b_{j_\ell}\in B^*\mid \mbox{$0\leqslant \ell\leqslant n-1$ and $1\leqslant j_1<\cdots<j_\ell\leqslant n-1$}\}. 
$$
Then, we have: 

\begin{lemma}\label{can+}
Let $w\in B^*$. Then, there exists a word $w'\in W^+$ such that $|w'|\<|w|$ and $w\sim_{R^+}w'$. 
\end{lemma}

\section{Main result}

In this section, we prove that the presentation $\langle A\cup B\mid R \rangle$ defines the monoid $\IO_n$ and, by means of substitutions (\textit{Titze transformations}),  
we also deduce another presentation for $\IO_n$ in terms of the minimal generators $a_{1},\ldots ,a_{n-2},b_{n-1}$. 

First, we set forth a series of lemmas. 

\begin{lemma}\label{R5'}
For $2\<j\<i\<n-1$, $b_ia_j\sim_R a_{j-1}b_{i-1}$. 
\end{lemma}
\begin{proof} 
Let us denote $\sim_R$ simply by $\sim$. 
Then, we have 
$$
b_ia_j\sim b_ib_ja_1\sim b_{j-1}b_ia_1\sim b_{j-1}a_i\sim a_ib_{j-1} \sim a_ia_{j-1}b_{n-1}\sim a_{j-1}a_{i-1}b_{n-1}\sim a_{j-1}b_{i-1}, 
$$
as required. 
\end{proof}

\begin{lemma}\label{biu}
Let $u\in A^*$ and $1\<i\<n-1$. Then, there exist $v\in A^*$ and $1\<k\<i\<n-1$ such that $b_iu\sim_R v$ or $b_iu\sim_R vb_k$. 
\end{lemma}
\begin{proof} 
For $|u|=0$, the lemma is immediate. 
So, assume by induction hypothesis that the lemma is valid for words $u\in A^*$ of length equal to $\ell$, for a given (fixed) $\ell\>0$. 
Let $u'\in A^*$ be such that $|u'|=\ell+1$. Then, $u'=ua_j$, for some $1\<j\<n-1$ and $u\in A^*$ such that $|u|=\ell$. 
Hence, by induction hypothesis, there exist $v\in A^*$ and $1\<k\<i\<n-1$ such that $b_iu\sim_R v$ or $b_iu\sim_R vb_k$.
If $b_iu\sim_R v$, then  $b_iu'=b_iua_j\sim_R va_j\in A^*$. 
On the other hand, if $b_iu\sim_R vb_k$, then 
$$
b_iu'=b_iua_j\sim_R vb_ka_j \sim_R 
\left\{
\begin{array}{ll}
va_jb_k & \mbox{if $1\<k<j\<n-1$}\\
va_{j-1}b_{k-1} & \mbox{if $2\<j\<k\<n-1$}\\
va_k & \mbox{if $j=1$},  
\end{array}
\right.
$$
as required. 
\end{proof}

Since we can write any word $w\in (A\cup B)^*$ in the form 
$$
w=u_0b_{i_1}u_1\cdots b_{i_k}u_k,
$$
with $1\<i_1\<\cdots\<i_k\<n-1$, $u_0,u_1,\ldots,u_k\in A^*$ and $k\>0$, by Lemma \ref{biu}, it is easy to conclude: 

\begin{corollary}\label{uv}
Let $w\in (A\cup B)^*$. Then, there exist $u\in A^*$ and $v\in B^*$ such that $w\sim_R uv$. 
\end{corollary}

\begin{lemma}\label{bn-1}
Let $v\in B^+$. Then, there exist $u\in A^+$ and $1\<\ell\<n-1$ such that $v\sim_R ub_{n-1}^\ell$. 
\end{lemma}
\begin{proof} 
By Lemma \ref{can+}, we can take $1\leqslant \ell\leqslant n-1$ and $1\leqslant j_1<\cdots<j_\ell\leqslant n-1$
such that $v\sim_{R^+} b_{j_1}\cdots b_{j_\ell}$. 
Then, denoting $\sim_R$ simply by $\sim$, we have 
\begin{align*}
v\sim b_{j_1}b_{j_2}\cdots b_{j_\ell} \sim a_{j_1} b_{n-1}b_{j_2}\cdots b_{j_\ell} \sim a_{j_1} b_{j_2-1}\cdots b_{j_\ell-1}b_{n-1}
\sim a_{j_1} a_{j_2-1}b_{n-1}b_{j_3-1}\cdots b_{j_\ell-1}b_{n-1}\sim  \\
\sim a_{j_1} a_{j_2-1}b_{j_3-2}\cdots b_{j_\ell-2}b_{n-1}^2 \sim\cdots\sim 
a_{j_1} a_{j_2-1}a_{j_3-2}\cdots a_{j_\ell-\ell+1}b_{n-1}^\ell, 
\end{align*}
as required. 
\end{proof} 

\begin{lemma}\label{aibn-1}
For $1\<\ell\<n-1$ and $n-\ell\<i\<n-1$, $a_ib_{n-1}^\ell\sim_R b_{n-1}^\ell$. 
\end{lemma}
\begin{proof} 
Denote $\sim_R$ simply by $\sim$. 
If $\ell=1$, then $i=n-1$ and we immediately have $a_{n-1}b_{n-1}\sim b_{n-1}$. 
So, suppose that $\ell\>2$. Then, we get  
$$
a_ib_{n-1}^\ell = a_ib_{n-1}b_{n-1}^{\ell-1}\sim b_ib_{n-1}^{\ell-1} =  b_ib_{n-1}b_{n-1}^{\ell-2} 
\sim b_{n-1}b_{i+1}b_{n-1}^{\ell-2} \sim \cdots \sim b_{n-1}^tb_{i+t}b_{n-1}^{\ell-t-1}, 
$$
for $1\< t\< n-i-1$ (notice that $n-\ell\< i$ implies that $t\<\ell-1$). 
In particular, 
$$
a_ib_{n-1}^\ell \sim b_{n-1}^{n-i-1}b_{n-1}b_{n-1}^{\ell-n+i}=b_{n-1}^\ell,
$$ 
as required. 
\end{proof}

Now, for $1\<\ell\<n-1$, let 
$$
W_\ell=\{a_{i_1}\cdots a_{i_k}b_{n-1}^\ell\mid \mbox{$0\<k\<n-1$ and $1\<i_k<\cdots<i_1\<n-\ell+k-2$}\}.
$$
Then, we have:

\begin{lemma}\label{well}
Let $u\in A^*$ and $v\in B^+$. Then, there exist $1\<\ell\<n-1$ and $w\in W_\ell$ such that $uv\sim_R w$. 
\end{lemma}
\begin{proof} 
Denote $\sim_R$ simply by $\sim$. 
By Lemma \ref{bn-1}, there exist $u'\in A^+$ and $1\<\ell\<n-1$ such that $v\sim u'b_{n-1}^\ell$. 
Hence, $uv \sim uu'b_{n-1}^\ell$. 
Then, by Lemma \ref{can-}, there exist $0\<k\<n-1$ and $1\<i_k<\cdots<i_1\<n-1$ such that 
$uu'\sim a_{i_1}\cdots a_{i_k}b_{n-1}^\ell$. Hence, $uv\sim a_{i_1}\cdots a_{i_k}b_{n-1}^\ell$. 

If $i_1\<n-\ell+k-2$, then $uv\sim a_{i_1}\cdots a_{i_k}b_{n-1}^\ell\in W_\ell$. 

So, suppose that $i_1>n-\ell+k-2$. Then, it is easy to show that, there exists $1\<t\<k$ such that $i_t>n-\ell+k-t-1>i_{t+1}$, where (if $t=k$) $i_{k+1}=-1$. 
Next, observe that 
\begin{align*}
a_{i_1}a_{i_2}\cdots a_{i_k} \sim a_{i_2}a_{i_1-1}a_{i_3}\cdots a_{i_k} \sim a_{i_2}\cdots a_{i_k}a_{i_1-k+1}
\sim a_{i_3}\cdots a_{i_k}a_{i_2-k+2}a_{i_1-k+1} \sim \\ 
\sim \cdots \sim a_{i_{t+1}}\cdots a_{i_k}a_{i_t-k+t}a_{i_{t-1}-k+t-1}\cdots a_{i_1-k+1}, 
\end{align*}
$n-\ell\< i_t-k+t\<i_{t-1}-k+t-1\<\cdots\<i_1-k+1$, since $i_t>n-\ell+k-t-1$ and $i_t<\cdots<i_1$, and $i_{t+1}\<n-\ell+(k-t)-2$. 
Thus, by Lemma \ref{aibn-1}, we have 
$$
uv\sim a_{i_1}\cdots a_{i_k}b_{n-1}^\ell \sim a_{i_{t+1}}\cdots a_{i_k}b_{n-1}^\ell \in W_\ell, 
$$
as required. 
\end{proof}

Let $W=W^-\cup W_1\cup\cdots\cup W_{n-1}$. Therefore, 
as an immediate consequence of Corollary \ref{uv} and Lemmas \ref{can-} and \ref{well}, we obtain: 

\begin{corollary}\label{can}
Let $w\in (A\cup B)^*$. Then, there exists $w'\in W$ such that $w\sim_R w'$. 
\end{corollary}

Now, observe that 
$$
|W|=|W^-|+\sum_{\ell=1}^{n-1}|W_\ell| = 2^{n-1} + \sum_{\ell=1}^{n-1}\sum_{k=0}^{n-1}\binom{\min\{n-\ell+k-2,n-1\}}{k} = 
\sum_{\ell=0}^{n-1}\sum_{k=0}^{n-1}\binom{\min\{n-\ell+k-2,n-1\}}{k}
$$
and it is not difficult to show that $\sum_{\ell=0}^{n-1}\sum_{k=0}^{n-1}\binom{\min\{n-\ell+k-2,n-1\}}{k}=(n+1)2^{n-2}$, whence $|W|=|\IO_n|$. 
Therefore, we are in a position to apply Proposition \ref{ruskuc} to conclude our main result: 

\begin{theorem}\label{pres}
For $n\>1$, the monoid $\IO_n$ is defined by the presentation $\langle A\cup B\mid R\rangle$ on $2n-2$ generators and  $\frac{1}{2}(3n^2-n-2)$ relations. 
\end{theorem} 

\medskip

Taking into account the equalities $a_{n-1}=b_{n-1}a_1$ and $b_i=a_ib_{n-1}$, for $1\<i\<n-2$, by making these substitutions into the set of relations $R$ 
and eliminating, from among those relations obtained, both the immediately trivial and the trivially deducible, we obtain the following set $R'$ of $\frac{1}{2}(3n^2-7n+4)$ relations on the alphabet $a_1,\ldots,a_{n-2},b_{n-1}$: 
\begin{description}
\item ($R'_2$) $a_ia_j=a_{j+1}a_i$,~ $1\<i\<j\<n-3$; 
\item\qquad $a_ia_{n-2}=b_{n-1}a_1a_i$,~ $1\<i\<n-2$; 

\item ($R'_4$) $a_ib_{n-1}a_{j+1}b_{n-1}=a_{j+1}b_{n-1}a_{i+1}b_{n-1}$,~ $1\<i\<j\<n-3$; 
\item\qquad $a_ib_{n-1}^2=b_{n-1}a_{i+1}b_{n-1}$,~ $1\<i\<n-3$; 
\item\qquad $a_{n-2}b_{n-1}^2=b_{n-1}^2$; 

\item ($R'_5$) $a_ib_{n-1}a_j=a_ja_ib_{n-1}$,~ $1\<i<j\<n-2$; 
\item\qquad $a_ib_{n-1}^2a_1=b_{n-1}a_1a_ib_{n-1}$,~ $1\<i\<n-2$; 

\item ($R'_6$) $a_ib_{n-1}a_1=a_i$,~ $ 1\<i\<n-2$; 
\item ($R'_7$) $b_{n-1}a_1b_{n-1}=b_{n-1}$. 
\end{description}

Therefore, we also have: 

\begin{corollary}\label{pres2}
For $n\>3$, the monoid $\IO_n$ is defined by the presentation $\langle a_1,\ldots,a_{n-2},b_{n-1} \mid R'\rangle$ on $n-1$ generators and  $\frac{1}{2}(3n^2-7n+4)$ relations. 
\end{corollary}

%\section*{Acknowledgement} 

%%%%%% 

{\small \sf  
\noindent{\sc V\'\i tor H. Fernandes},
Center for Mathematics and Applications (NOVA Math)
and Department of Mathematics,
Faculdade de Ci\^encias e Tecnologia,
Universidade Nova de Lisboa,
Monte da Caparica,
2829-516 Caparica,
Portugal;
e-mail: vhf@fct.unl.pt.
} 

\end{document}